\documentclass{kms-j}
\usepackage[margin=3cm]{geometry}
\usepackage{setspace}   
\usepackage{mathrsfs}
\usepackage{amssymb}
\usepackage{amsmath}
\theoremstyle{plain}
\newtheorem{theorem}{Theorem}[section]
\newtheorem{proposition}[theorem]{Proposition}
\newtheorem{lemma}[theorem]{Lemma}

\newtheorem{corollary}[theorem]{Corollary}

\theoremstyle{definition}
\newtheorem{definition}[theorem]{Definition}

\theoremstyle{remark}
\newtheorem{remark}[theorem]{Remark}

\begin{document}

\title[Local well-posedness and global existence]
{Local well-posedness and global existence for a multi-component Novikov equation}

\author{Zhigang Li}
\address{College of Sciences\\China University of Mining and Technology\\Beijing, 100083, People's Republic of China}
\email{lzgcumtb@163.com}

\author{Yuxi Hu*}
\address{College of sciences\\China University of Mining and Technology\\Beijing, 100083, People's Republic of China}
\email{yxhu86@163.com}

\author{Xinglong Wu}
\address{School of Science\\Wuhan University of Technology\\Wuhan, 430070, People's Republic of China}
\email{wxl8758669@aliyun.com}

\begin{abstract}
Considered herein is a multi-component Novikov equation, which admits bi-Hamiltonian structure, infinitely many conserved quantities and peaked solutions. In this paper, we deduce two blow-up criteria for this system and global existence for some two-component case in $H^s$. Finally we verify that the system possesses peakons and periodic peakons.

Correspondence should be addressed to Yuxi Hu; hu-yuxi@163.com
\end{abstract}

\maketitle

\section{Introduction}
In this paper we consider the following multi-component Novikov equations [2]:
\begin{equation}
\begin{cases}
m_{it}=\sum_{j=1}^N\left(-2m_iu_{jx}v_j-m_iu_jv_{jx}-m_{ix}u_jv_j-u_{ix}m_jv_j+u_im_jv_{jx}\right),\\
n_{it}=\sum_{j=1}^N\left(-2n_iu_jv_{jx}-n_iu_{jx}v_j-n_{ix}u_jv_j-v_{ix}n_ju_j+v_in_ju_{jx}\right),\\
m_i=u_i-u_{i,xx},\ \ \ n_i=v_i-v_{i,xx}.
\end{cases}
\end{equation}

Many scholars have focused on Camassa-Holm type equations over two decades. Most of these equations admit zero-curvature equation, bi-Hamiltonian structure, and infinite many conservation laws. CH type equations have two remarkable mathematical features : peaked solutions and wave-breaking phenomenon, this is the reason why CH type equations are brought into focus.

For $N=1$, (1) is reduced to Geng-Xue (GX) equation [3],
\begin{equation}
\begin{cases}
m_t+uvm_x+3u_xvm=0,\\
n_t+uvn_x+3uv_xn=0,\\
m=u-u_{xx},\ \ \ n=v-v_{xx}.
\end{cases}
\end{equation}
This is a two-component CH type equation which is derived by Geng and Xue, and the authors also calculated this equation also admits multi-peakons and infinite many conservation laws, but the bi-Hamiltonian structure was constructed by Li and Liu [4]. In fact, (2) is an extension of Novikov (Nov) equation[5] if we take $u=v$,
$$m_t+u^2m_x+3uu_xm=0, \ \ \ m=u-u_{xx}.$$
Novikov equation was derived by Novikov, used the perturbative symmetry approach in classification of nonlocal PDES with three order nonlinearity. It should be note that Novikov equation is not symmetrical, which means $(u,x) \nrightarrow (-u,-x)$.

Another interesting reduction for (2) is DP equation if we take $v=1$,
$$m_t+um_x+3u_xm=0, \ \ \ m=u-u_{xx}.$$
The DP equation was derived by Degasperis and Procesi [6] by applying the method of asymptotic integrability to a three order dispersive PDE. A big feature for DP equation is shock peakon [7,8]
$$u(t,x)=-\frac{1}{t+k} \operatorname{sgn}(x) e^{-|x|}.$$

In recent years, people have generated CH equation to some multi-component sense. Holm and Ivanov derived a multi-component system (CH(n,k)) [9],
\begin{equation*}
\begin{cases}
\partial_t q_p + \displaystyle\sum_{j=0}^{k-1}\left(\partial_x q_{p+j} +q_{p+j} \partial_x\right) u_j -\frac{1}{2}q_{p+k,x}=0,\\
q_{n+1}=q_{n+2}=\cdots=q_{n+k}=0,
\end{cases}
\end{equation*}
where $u_0,\cdots,u_{k-1}$ and $q_1,\cdots,q_k$ satisfies the following differential relations
\begin{equation*}
\begin{cases}
q_{n-r,t}=\displaystyle -\sum_{s=\max (0,r-k)}^r \left(\partial_x q_{n-s} +q_{n-s} \partial_x\right) u_{r-s}, r=0,1,\cdots,n-1,\\
0=\frac{1}{2}\left(\partial_x-\partial_x^3\right)u_r+\displaystyle \sum_{s=1}^{\min (n,k-r)}\left(\partial_x q_s +q_s \partial_x\right)u_{r+s}, r=0,1,\cdots,k-1,\\
0=\left(\partial_x-\partial_x^3\right)u_k.
\end{cases}
\end{equation*}

Obviously, CH(1,1) is Camassa-Holm (CH) equation [10-12],
$$m_t+2u_xm+um_x=0,\ \ \ m=u-u_{xx}.$$
This equation was derived by Camassa and Holm [10] by using an asymptotic expansion in the Hamiltonian for Euler¡¯s equations directly in the shallow water regime. While, this equation was firstly found by Fokas and Fuchssteiner [11] when the authors used the KdV recursion operator. Olver also obtain such equation by taking approach of tri-Hamiltonian duality on KdV equation [12]. For $n=2,k=1$, CH(2,1) is the two-component (2CH) equation [13],
\begin{equation*}
\begin{cases}
m_t+\left(\partial_x m + m\partial_x\right)u \mp \rho \rho_x=0, \\
\rho_t+(u\rho)_x=0,\ \ \ m=u-u_{xx}.
\end{cases}
\end{equation*}
This equation was also founded by Olver [12] through such technique on Ito system, but it was firstly derived rigorously by Constantin and Ivanov [13]. 2CH equation may be regarded as a model of shallow water waves, where $\rho$ denotes the density. For more details of 2CH, one can refer to [14-16] and the references therein.

Xia and Qiao has offered another multi-component generation of CH equation (CH(N,H))[17]
\begin{equation*}
\begin{cases}
m_{jt}=(m_jH)_x+m_jH+\dfrac{1}{(N+1)^2}\displaystyle\sum_{i=1}^N\left[m_i(u_j-u_{jx})(v_i+v_{ix})+m_j(u_i-u_{ix})(v_i+v_{ix})\right],\\
n_{jt}=(n_jH)_x-n_jH-\dfrac{1}{(N+1)^2}\displaystyle\sum_{i=1}^N\left[n_i(u_i-u_{ix})(v_j+v_{jx})+n_j(u_i-u_{ix})(v_i+v_{ix})\right],\\
m_j=u_j-u_{jxx},\ \ \ n_j=v_j-v_{jxx},\ \ \ 1 \le j \le N.
\end{cases}
\end{equation*}
where $H$ is an arbitrary smooth function of $u$, $v$, $u_x$ and $v_x$. It is easy to see, CH(N,H) is different with CH(n,k) while each equation has the same structure. When $N=1$ [18], with a special taken of $H$-a polynomial function, CH(1,H) can be reduced to some certain equations. For example, if we take $H=\frac{1}{2}(u-u_x)(v+v_x)$, it is reduced to SQQ equation [19],
\begin{equation*}
\begin{cases}
m_t=\dfrac{1}{2}\left[m(u-u_x)(v+v_x)\right]_x,\\
n_t=\dfrac{1}{2}\left[n(u-u_x)(v+v_x)\right]_x,\\
m=u-u_{xx},\ \ \ n=v-v_{xx}.
\end{cases}
\end{equation*}
Furthermore, if we take $v=1$, and $H=u$, then CH(1,H) is CH equation; for $v=2u$, $H=(u^2-u_x^2)$, CH(1,H) is reduced to FORQ (mCH) equation [12,20,21],
$$m_t+[m(u^2-u_x^2)]_x=0,\ \ \ m=u-u_{xx}.$$
Zhang and Yin has studied some qualitative analysis of CH(N,H) [22], such as well-posedness in Besov space, and blow-up phenomenon.

It should be noticed that system (1) can not be reduced by CH(N,H). The well-posedness and global existence of Novikov equation has been studied by Wu, Yin [24,25], the Cauchy problem for GX equation was studied by Mi, Mu and Tao [26]. Very recently, the well-posedness for (1) has been proved by Mi, Guo and Luo [27]. To our best knowledge, the local well-posedness in critical Besov space, blow-up criteria and global existence have not be studied yet. Inspired by Wu and Yin, the aim of this paper is to consider the local well-posedness in critical sense and blow-up criteria of system (1) by using the transport equation technique of Besov space, and global existence for some two-component reductions, we will show that the multi-component Novikov equation also admits peaked and periodic peaked traveling wave solutions. The outline of this paper is as follows: In section 2, we recall Littlewood-Paley theory and some lemmas which are very important for the later proofs. In section 3, we prove the local well-posedness in critical Besov space, by using transport equation theory in Besov space. We show the blow-up criteria for (1) and some two-component reduction cases in section 4. The global existence of Geng-Xue equation in $H^s(\mathbb{R})$ $s>\frac{5}{2}$ is proved in section 5. And we show that system (1) admits peakons and periodic peakons in last section.

\section{Preliminaries}

\begin{proposition}$[1]$
(Littlewood-Paley decomposition)Let\ $\mathcal{C}$ be the annulus\ $\{\xi \in \mathbb{R}^d|\ \tfrac{3}{4} \leq |\xi| \leq \tfrac{8}{3}\}$.\\There exist radial functions $\chi$ and $\varphi$ , valued in the interval $[0,1]$, belonging respectively to $\mathcal{D}(B(0,\frac{4}{3}))$, and such that
$$\forall \xi \in \mathbb{R}^d,\chi(\xi)+\sum_{j \geq 0}\varphi(2^{-j}\xi)=1,$$
$$\forall \xi \in \mathbb{R}^d\backslash\{0\},\sum_{j \in \mathbb{Z}} \varphi(2^{-j}\xi)=1,$$
$$|j-j'| \geq 2 \Rightarrow Supp\ \varphi(2^{-j}\cdot)\cap Supp\ \varphi(2^{-j'}\cdot) = \varnothing,$$
$$j \geq 1 \Rightarrow Supp\ \chi \cap Supp\ \varphi(2^{-j}\cdot) = \varnothing,$$
the set $\widetilde{\mathcal{C}}=B(0,\frac{2}{3})+\mathcal{C}$ is an annulus, and we have
$$|j-j'| \geq 5 \Rightarrow 2^{j'}\mathcal{C} \cap 2^j\mathcal{C} = \varnothing.$$
Further, we have
$$\forall \xi \in \mathbb{R}^d, \frac{1}{2} \leq \chi^2(\xi)+\sum_{j \geq 0}\varphi^2 (2^{-j}\xi) \leq 1,$$
$$\forall \xi \in \mathbb{R}^d\backslash\{0\},\frac{1}{2} \leq \sum_{j \in \mathbb{Z}} \varphi^2(2^{-j}\xi) \leq 1.$$
\end{proposition}

Denote by $\mathcal{F}$ the Fourier transform and by $\mathcal{F}^{-1}$ its inverse. From now on ,we write $h=\mathcal{F}^{-1}\varphi$ and $\tilde{h}=\mathcal{F}^{-1}\chi$.The nonhomogeneous dyadic blocks $\Delta_j$ are defined by
$$\Delta_ju = 0 \ \ if\ \ j \leq -2,\ \Delta_{-1}u=\chi(D)u=\int_{\mathbb{R}^d}\tilde{h}(y)u(x-y)dy,$$
$$and,\ \Delta_ju=\varphi(2^{-j}D)u=2^{jd}\int_{\mathbb{R}^d}h(2^jy)u(x-y)dy\ \ if \ \ j \geq 0, $$
$$S_ju=\ \sum_{j' \le j-1}\Delta_{j'}u.$$

The nonhomogeneous Besov spaces are denoted by $B_{p,r}^{s}(\mathbb{R}^d)$
$$B_{p,r}^{s}(\mathbb{R}^d)=\{u \in \mathcal{S}'| \|u\|_{B_{p,r}^{s}}=\left(\sum_{j \geq -1}2^{rjs}\|\Delta_ju\|_{L^p}^r\right)^{\tfrac{1}{r}} \le \infty\}.$$

\begin{proposition}
$[1]$The following properties hold.\\
(i)Density: if $p,r \le \infty$, then $\mathcal{S}(\mathbb{R}^d)$ is dense in $B_{p,r}^{s}(\mathbb{R}^d)$,where $\mathcal{S}$ denotes the Schwartz space.\\
(ii)Embeddings: if $p_1 \leq p_2$ and $r_1 \leq r_2$, then $B_{p_1,r_1}^s \hookrightarrow B_{p_2,r_2}^{s-d(\tfrac{1}{p_1}-\tfrac{1}{p_2})}$.If $s_1 < s_2$, $1 \le p \le +\infty$ and $1 \le r_1,r_2 \le +\infty$, then the embedding $B_{p,r_2}^{s_2} \hookrightarrow B_{p,r_1}^{s_1}$ is locally compact.\\
(iii)Algebraic properties: for $s>0$, $B_{p,r}^s \cap L^{\infty}$ is an algebra. Moreover, $B_{p,r}^s$ is an algebra $\Leftrightarrow$ $B_{p,r}^s \hookrightarrow L^{\infty} \Leftrightarrow s>\tfrac{d}{p}$ or $s \ge \tfrac{d}{p}$ and $r=1$.\\
(iv)Fatou property: if $(u^{(n)})_{n \in \mathbb{N}}$ is a bounded sequence of $B_{p,r}^s $ which tends to u in $\mathcal{S}'$, then $u \in B_{p,r}^s$ and $$\|u\|_{B_{p,r}^s } \le \liminf_{n\rightarrow\infty}\|u^{(n)}\|_{B_{p,r}^s}.$$
(v)$B_{2,2}^s$ is coincides with the Sobolev space $H^s$.
\end{proposition}

\begin{lemma}
$[1]$(1-D Moser-type estimates). Assume that $1 \le p,r \le +\infty$, the following estimates hold:\\
(i) for $s>0, \|fg\|_{B_{p,r}^s} \le C\left(\|f\|_{B_{p,r}^s}\|g\|_{L^{\infty}}+\|g\|_{B_{p,r}^s}\|f\|_{\infty}\right);$\\
(ii) for $s_1 \le \tfrac{1}{p}, s_2 > \tfrac{1}{p}(s_2 \ge \tfrac{1}{p}\ if\ r=1)$,and $s_1+s_2>0$,
$$\|fg\|_{B_{p,r}^{s_1}} \le C\|f\|_{B_{p,r}^{s_1}}\|g\|_{B_{p,r}^{s_2}},$$
where the constanc C is independent of f and g.
\end{lemma}

\begin{lemma}
$[1]$For $s \ge 0$, the following estimates hold:
$$\|fg\|_{H^s} \le C\left(\|f\|_{H^s}\|g\|_{L^{\infty}}+\|f\|_{L^{\infty}}\|g\|_{H^s}\right),$$
$$\|f \partial_x g\|_{H^s} \le C\left(\|f\|_{H^{s+1}}\|g\|_{L^{\infty}}+\|f\|_{L^{\infty}}\|\partial_x g\|_{H^s}\right),$$
where the constant C is independent of f and g.
\end{lemma}

\begin{lemma}
$[1]$(Osgood lemma) Let $\rho$ be a measurable function from $[t_0,T]$ to $[0,a]$, $\gamma$ is a locally integrable function from $[0,T]$ to $\mathbb{R}^+$, and $\mu$ is a continuous and nondecreasing function from $[0,a]$ to $\mathbb{R}^+$. Assume for some nonnegative real nunber $c$, the function $\rho$ satisfies
$$\rho(t) \le c+\int_{t_0}^t\gamma(s)\mu(\rho(s))ds,\ \ \ for \ \ \ a.e. \ \ \ t\in[t_0,T].$$
If $c$ is positive, then we have for a.e. $t\in[t_0,T]$
$$-\mathcal{N}(\rho(t))+\mathcal{N}(c) \le \int_{t_0}^t\gamma(s)ds \ \ \ with \ \ \ \mathcal{N}(x)=\int_x^a\frac{d\tau}{\mu(\tau)}.$$
If $c=0$ and $\mu$ satisfies $\displaystyle\int_0^a\frac{d\tau}{\mu(\tau)} = +\infty$, then $\rho=0$ a.e.
\end{lemma}

Now we introduce a priori estimates for the following transport equation.
\begin{equation}
\begin{cases}
\partial_t f+v\cdot \nabla f =g,\\
f|_{t=0}=f_0.
\end{cases}
\end{equation}

\begin{lemma}
$[1]$(A priori estimates in Besov spaces).Let $1 \le p \le p_1 \le \infty$, $1 \le r \le \infty$, $s \ge -d\cdot min(\tfrac{1}{p_1},\tfrac{1}{p'})$. For the solution $ f \in L^{\infty}([0,T];B_{p,r}^s(\mathbb{R}^d))$ of (3) with velocity $v, \nabla v \in  L^1([0,T];B_{p,r}^s(\mathbb{R}^d)\cap L^{\infty}(\mathbb{R}^d))$, initial data $f_0 \in B_{p,r}^s(\mathbb{R}^d)$ and $g\in L^1([0,T];B_{p,r}^s(\mathbb{R}^d))$, we have
$$\|f(t)\|_{B_{p,r}^s} \le \|f_0\|_{B_{p,r}^s}+\int_0^t\left(\|g(t')\|_{{B_{p,r}^s}}+CV'_{p_1}(t')\|f(t')\|_{{B_{p,r}^s}}\right)dt',$$
$$\|f(t)\|_{L_t^{\infty}(B_{p,r}^s)} \le \left(\|f_0\|_{B_{p,r}^s}+\int_0^t\exp(-CV_{p_1}(t'))\|g(t')\|_{{B_{p,r}^s}}dt'\right)\exp(CV_{p_1}),$$
where
\begin{equation*}
V_{p_1}(t)=\begin{cases}
\displaystyle\int_0^t\|\nabla v\|_{B_{p_1,\infty}^{\frac{d}{p_1}} \cap L^{\infty}} &,\  if\ \ \  s<1+\tfrac{d}{p_1},\\
\displaystyle\int_0^t\|\nabla v\|_{B_{p_1,r}^{s-1}}&,\  if\ \ \  s>1+\tfrac{d}{p_1}\  or\  s=1+\tfrac{d}{p_1}, r=1,\\
\displaystyle\int_0^t\|\nabla v\|_{B_{p_1,1}^{\frac{d}{p_1}}}  &,\  if\ \ \  s=-d\cdot \min\{\frac{1}{p'},\frac{1}{p_1}\}, r=\infty.
\end{cases}
\end{equation*}
\end{lemma}

For $d=1$ and $0<s<1$, we have the following a priori estimates.

\begin{lemma}
$[1]$(A priori estimates in Sobolev spaces). Let $0<s\le1$. Assume that $ f \in L^{\infty}([0,T];H^s(\mathbb{R}))$ is the solution of(1) with velocity $v, \partial_x v \in  L^1([0,T];L^{\infty}(\mathbb{R}))$, initial data $f_0 \in H^s(\mathbb{R})$ and $g\in L^1([0,T];H^s(\mathbb{R}))$, we have
$$\|f(t)\|_{H^s} \le \|f_0\|_{H^s}+\int_0^t\left(\|g(t')\|_{H^s}+CV'(t')\|f(t')\|_{H^s}\right)dt',$$
$$\|f(t)\|_{H^s} \le \left(\|f_0\|_{H^s}+\int_0^t\exp(-CV(t'))\|g(t')\|_{{H^s}}dt'\right)\exp(CV(t)),$$
where $V(t)=\displaystyle\int_0^t\|v\|_{W^{1,\infty(\mathbb{R})}}d\tau$ and C is a constant depending only on s.
\end{lemma}

\begin{lemma}
$[23]$(Kato-Ponce commutator estimates). If $s>0$, $f\in H^s(\mathbb{R}) \cap W^{1,\infty}(\mathbb{R})$, $g\in H^{s-1}(\mathbb{R}) \cap L^{\infty}(\mathbb{R})$, and denote that $\Lambda^s = (1-\Delta)^{\tfrac{s}{2}}$, where $\Delta$ is Laplacian then
$$\|\Lambda^s(fg)-f\Lambda^sg\|_{L^2} \le C\left(\|\Lambda^sf\|_{L^2}\|g\|_{L^{\infty}}+\|f_x\|_{L^{\infty}}\|\Lambda^{s-1}g\|_{L^2}\right).$$
\end{lemma}

\section{local well-posedness}
Mi, Guo and Luo has just studied the local well-posedness of (1) in [27], in this section, we study the local well-posedness in critical Besov space. For the convenience of proof, we rewrite (1) into another equivalent form

\begin{equation}
\begin{cases}
M_t+a(U,U_x)M_x=B(U,U_x)M,\\
M|_{t=0}=M_0,
\end{cases}
\end{equation}

where $M=(m_1,\cdots,m_N,n_1,\cdots,n_N)^T$, $U=(u_1,\cdots,u_N,v_1,\cdots,v_N)^T$,
\begin{equation*}
a(U,U_x)=\sum_{j=1}^Nu_jv_j,\ \ \ B(U,U_x)=\left(\begin{matrix}
B_{11} & 0\\
0 & B_{22}
\end{matrix}\right),
\end{equation*}
\begin{equation*}
B_{11}=\left(\begin{matrix}
u_1v_{1x}-u_{1x}v_1 & \cdots & u_1v_{Nx}-u_{1x}v_N\\
\vdots & \vdots & \vdots\\
u_Nv_{1x}-u_{Nx}v_1 & \cdots & u_Nv_{Nx}-u_{Nx}v_N\\
\end{matrix}\right)-\sum_{j=1}^N\left(2u_{jx}v_j+u_jv_{jx}\right)I_{N\times N},
\end{equation*}
\begin{equation*}
B_{22}=\left(\begin{matrix}
v_1u_{1x}-v_{1x}u_1 & \cdots & v_1u_{Nx}-v_{1x}u_N\\
\vdots & \vdots & \vdots\\
v_Nu_{1x}-v_{Nx}u_1 & \cdots & v_Nu_{Nx}-v_{Nx}u_N\\
\end{matrix}\right)-\sum_{j=1}^N\left(2v_{jx}u_j+v_ju_{jx}\right)I_{N\times N}.
\end{equation*}

The local well-posedness of (1) is stated as follows
\begin{theorem}
Suppose $M_0\in (B_{2,1}^{\frac{1}{2}})^{2n}$. Then there exists some $T>0$, such that system (4) has a unique solution in $[C([0,T);B_{2,1}^{\frac{1}{2}})\bigcap C^1([0,T);B_{2,1}^{-\frac{1}{2}})]^{2n}$ .
\end{theorem}

\begin{remark}
Since the Cauchy problem of the CH equation is ill-posed in Besov space $B_{2,\infty}^{\frac{1}{2}}$, hence (1) is not locally well-posed in $B_{2,\infty}^{\frac{1}{2}}$. By embedding theorem, for $s>\frac{1}{2}$, $H^s \hookrightarrow B_{2,1}^{\frac{1}{2}} \hookrightarrow B_{2,\infty}^{\frac{1}{2}}$,
the Besov space $B_{2,\infty}^{\frac{1}{2}}$ is almost critical for the well-posedness of (1).
\end{remark}

\begin{proof}
In order to prove Theorem 3.1, we proceed as the following steps.

\textbf{Step 1.}
Firstly, we construct approximate solutions which are smooth solutions of some linear equations.

Starting for $M^0=M_0$ we define by induction sequences $(M^n)_{n \in \mathbb{N}}$ by solving the following linear transport equations:
\begin{equation}
\begin{cases}
M_t^{n+1}+a^nM_x^{n+1}=B^nM^n,\\
M|_{t=0}^{n+1}=S_{n+1}M_0,
\end{cases}
\end{equation}
where $M^n=(m_1^n,...m_N^n,n_1^n,...n_N^n)^T$, $U^n=(u_1^n,...u_N^n,v_1^n,...v_N^n)^T$, $a^n=a(U^n,U^n_x)$ and $B^n=B(U^n,U^n_x)$.

\textbf{Step 2.}
Next, we show that, for some fixed positive number $T$, the sequences $(M^n)_{n\in\mathbb{N}}$ is uniformly bounded in $[C([0,T);B_{2,1}^{\frac{1}{2}})\bigcap C^1([0,T);B_{2,1}^{-\frac{1}{2}})]^{2n}$, we define that $V(t)=\displaystyle\int_0^t\|a^n_x\|_{B_{2,1}^{-\frac{1}{2}}}d\tau$. By virtue of $B_{2,1}^{\frac{1}{2}}$ is an algebra, we infer that

\begin{equation}
\begin{aligned}
\|B^nM^n\|_{B_{2,1}^{\frac{1}{2}}} &\le \|B^n\|_{B_{2,1}^{\frac{1}{2}}}\|M^n\|_{B_{2,1}^{\frac{1}{2}}}\\
&\le C(\|U^n\|_{B_{2,1}^{\frac{1}{2}}}^2+\|U^n_x\|_{B_{2,1}^{\frac{1}{2}}}^2)\|M^n\|_{B_{2,1}^{\frac{1}{2}}} \le C\|M^n\|_{B_{2,1}^{\frac{1}{2}}}^3.
\end{aligned}
\end{equation}
Applying Lemma 2.6 with $p_1=p=2$, we have
\begin{equation}
\begin{aligned}
\|M^{n+1}(t)\|_{{L_t^\infty}(B_{2,1}^{\frac{1}{2}})} &\le e^{CV(t)}\Big(\|S_{n+1}M_0\|_{B_{2,1}^{\frac{1}{2}}}+C\int_0^te^{-CV(\tau)}\|B^n(\tau)M^n(\tau)\|_{B_{2,1}^{\frac{1}{2}}}d\tau\Big)\\
&\le e^{C\int_0^t\|M^n(\tau)\|_{B_{2,1}^{\frac{1}{2}}}^2d\tau}\Big(\|M_0\|_{B_{2,1}^{\frac{1}{2}}}+C\int_0^te^{-C\int_0^\tau\|M^n(s)\|_{B_{2,1}^{\frac{1}{2}}}^2ds}\|M^n(\tau)\|_{B_{2,1}^{\frac{1}{2}}}^3d\tau\Big).
\end{aligned}
\end{equation}
With a similar argument in [24], we have the following estimation
\begin{equation}
\|M^{n+1}(t)\|_{B_{2,1}^{\frac{1}{2}}} \le \frac{\|M_0\|_{B_{2,1}^{\frac{1}{2}}}}{\sqrt{1-4C\|M_0\|_{B_{2,1}^{\frac{1}{2}}}^2 t}},
\end{equation}
which means $(M^n)_{n\in\mathbb{N}}$ is uniformly bounded in $(C([0,T];B_{2,1}^{\frac{1}{2}}))^{2n}$. Providing by (1), it is easy to check $\partial_t(M^{n+1})$ in $C([0,T);B_{2,1}^{\frac{1}{2}}))^{2n}$ is also uniformly bounded. Thus, $(M^n)_{n\in\mathbb{N}}$ is uniformly bounded in $[C([0,T);B_{2,1}^{\frac{1}{2}})\bigcap C^1([0,T);B_{2,1}^{-\frac{1}{2}})]^{2n}$.

\textbf{Step 3.}
In this part, we prove that $(M^n)_{n\in\mathbb{N}}$ is a Cauchy sequence in $C([0,T);B_{2,\infty}^{-\frac{1}{2}})^{2n}$. By virtual of (4), we have

\begin{equation}
\begin{aligned}
(M^{m+n+1}-M^{n+1})_t&+a^{m+n}(M^{m+n+1}-M^{n+1})_x=R_{m,n}=\\
&-(a^{m+n}-a^n)M^{n+1}_x+B^{m+n}(M^{m+n}-M^n)+(B^{m+n}-B^n)M^n,
\end{aligned}
\end{equation}
with the initial data
$$(M^{m+n+1}-M^{n+1})\Big|_{t=0}=(S_{m+n+1}-S_{n+1})M_0.$$

With a direct calculation, we have
\begin{equation}
\begin{aligned}
&\ \ \ \|(a^{m+n}-a^n)M_x^{n+1}\|_{B_{2,\infty}^{-\frac{1}{2}}}\\
&\le C\|a^{m+n}-a^n\|_{B_{2,\infty}^{\frac{1}{2}} \cap L^\infty}\|M_x^{n+1}\|_{B_{2,1}^{-\frac{1}{2}} }\\
&\le C\|U^{m+n}-U^n\|_{B_{2,\infty}^{\frac{3}{2}}}(\|U^{m+n}\|_{B_{2,\infty}^{\frac{3}{2}}}+\|U^n\|_{B_{2,\infty}^{\frac{3}{2}}})\|M^{n+1}\|_{B_{2,1}^{\frac{1}{2}}}\\
&\le C\|M^{m+n}-M^n\|_{B_{2,\infty}^{-\frac{1}{2}}}(\|M^{m+n}\|_{B_{2,\infty}^{-\frac{1}{2}}}+\|M^n\|_{B_{2,\infty}^{-\frac{1}{2}}})\|M^{n+1}\|_{B_{2,1}^{\frac{1}{2}}}\\
&\le C\|M^{m+n}-M^n\|_{B_{2,1}^{-\frac{1}{2}}}(\|M^{m+n}\|_{B_{2,1}^{\frac{1}{2}}}+\|M^n\|_{B_{2,1}^{\frac{1}{2}}})\|M^{n+1}\|_{B_{2,1}^{\frac{1}{2}}},
\end{aligned}
\end{equation}
here we used the fact that $U=(1-\partial_x^2)^{-1}M$, where $(1-\partial_x^2)^{-1}$ is a $S^{-2}$-multiplier, and paraproduct [28]
$$\|fg\|_{B_{2,\infty}^{-\frac{1}{2}}} \le\|f\|_{B_{2,\infty}^{\frac{1}{2}} \cap L^\infty}\|g\|_{B_{2,1}^{-\frac{1}{2}}}.$$
Similarly, we have the estimations for the remaining two terms
\begin{equation}
\begin{aligned}
&\ \ \ \|B^{m+n}(M^{m+n}-M^n)\|_{B_{2,\infty}^{-\frac{1}{2}}}\\
&\le C\|B^{m+n}\|_{B_{2,\infty}^{\frac{1}{2}}\cap L^\infty}\|(M^{m+n}-M^n)\|_{B_{2,1}^{-\frac{1}{2}}}\\
&\le C\|U^{m+n}\|_{B_{2,1}^{\frac{1}{2}}}\|U^{m+n}_x\|_{B_{2,1}^{\frac{1}{2}}}\|(M^{m+n}-M^n)\|_{B_{2,1}^{-\frac{1}{2}}}\\
&\le C\|(M^{m+n}-M^n)\|_{B_{2,1}^{-\frac{1}{2}}}\|M^{m+n}\|_{B_{2,1}^{\frac{1}{2}}}^2,
\end{aligned}
\end{equation}
\begin{equation}
\begin{aligned}
&\ \ \ \|(B^{m+n}-B^n)M^n\|_{B_{2,\infty}^{-\frac{1}{2}}}\\
&\le C\|B^{m+n}-B^n\|_{B_{2,\infty}^{\frac{1}{2}} \cap L^\infty}\|M^n\|_{B_{2,1}^{-\frac{1}{2}}}\\
&\le C\Big((\|U^{m+n}\|_{B_{2,\infty}^{\frac{1}{2}} \cap L^\infty}+\|U^n\|_{B_{2,\infty}^{\frac{1}{2}} \cap L^\infty})\|(U^{m+n}-U^n)_x\|_{B_{2,\infty}^{\frac{1}{2}}\cap L^\infty}\\
&+(\|U^{m+n}_x\|_{B_{2,\infty}^{\frac{1}{2}}\cap L^\infty}+\|U^n_x\|_{B_{2,\infty}^{\frac{1}{2}}\cap L^\infty})\|U^{m+n}-U^n\|_{B_{2,\infty}^{\frac{1}{2}}\cap L^\infty}\Big)\|M^n\|_{B_{2,1}^{\frac{1}{2}}}\\
&\le C\|M^{m+n}-M^n\|_{B_{2,1}^{-\frac{1}{2}}}(\|M^{m+n}\|_{B_{2,1}^{\frac{1}{2}}}\|M^n\|_{B_{2,1}^{\frac{1}{2}}}+\|M^n\|_{B_{2,1}^{\frac{1}{2}}}^2).
\end{aligned}
\end{equation}
In view of Lemma 2.6 to (4) and take advantage of (10)-(12), for all $t \in [0,T]$, we obtain
\begin{equation}
\begin{aligned}
&\ \ \ \|(M^{n+m+1}-M^{n+1})(t)\|_{B_{2,\infty}^{-\frac{1}{2}}}\\
&\le \|S_{n+m+1}M_0-S_{n+1}M_0\|_{B_{2,\infty}^{-\frac{1}{2}}}+C\int_0^tV'(\tau)\|(M^{n+m}-M^n)(\tau)\|_{B_{2,\infty}^{-\frac{1}{2}}}\\
&\ \ \ +\int_0^t\|\big((a^{m+n}-a^n)M_x^{n+1}+B^{n+m}(M^{n+m}-M^n)+(B^{m+n}-B^n)M^n\big)(\tau)\|_{B_{2,\infty}^{-\frac{1}{2}}}d\tau\\
&\le \|S_{n+m+1}M_0-S_{n+1}M_0\|_{B_{2,\infty}^{-\frac{1}{2}}}+C\int_0^t\|a^{n+m}\|_{B_{2,1}^{\frac{3}{2}}}\|(M^{n+m}-M^n)(\tau)\|_{B_{2,\infty}^{-\frac{1}{2}}}d\tau\\
&\ \ \ +C\int_0^t\|(M^{n+m}-M^n)(\tau)\|_{B_{2,1}^{-\frac{1}{2}}}(\|M^{n+m}(\tau)\|_{B_{2,1}^{\frac{1}{2}}}+\|M^n(\tau)\|_{B_{2,1}^{\frac{1}{2}}}+\|M^{n+1}\|_{B_{2,1}^{\frac{1}{2}}})^2d\tau.
\end{aligned}
\end{equation}
where $V(t)=\displaystyle\int_0^t\|a^{n+m}_x\|_{B_{2,1}^{\frac{1}{2}}}d\tau$. In view of the following interpolation inequality [29]
\begin{equation}
\|f\|_{B_{2,1}^{-\frac{1}{2}}} \le C\|f\|_{B_{2,\infty}^{-\frac{1}{2}}}\log \left(e+\frac{\|f\|_{B_{2,\infty}^{\frac{1}{2}}}}{\|f\|_{B_{2,\infty}^{-\frac{1}{2}}}}\right),
\end{equation}
and the auxiliary function
\begin{equation}
q(x)=\log(e+\theta)\log\frac{1}{x}-\log(e+\frac{\theta}{x})+\log(e+\theta),\ \ \ \theta>0.
\end{equation}
Since for $0<x \le 1$, $q'(x)<0$ and $q(1)=0$, we have
\begin{equation}
\log(e+\frac{\theta}{x}) \le \log(e+\theta)(1+\log\frac{1}{x}),
\end{equation}
throw (14), we have
\begin{equation}
\|(M^{n+m}-M^n)(\tau)\|_{B_{2,1}^{-\frac{1}{2}}} \le C\|(M^{n+m}-M^n)(\tau)\|_{B_{2,\infty}^{-\frac{1}{2}}}\left(1-\log\left(\|(M^{n+m}-M^n)(\tau)\|_{B_{2,\infty}^{-\frac{1}{2}}}\right)\right).
\end{equation}
As we have just proved that the sequence $(M^n)_{n\in\mathbb{N}}$ is uniformly bounded in $(C([0,T];B_{2,1}^{\frac{1}{2}}))^{2n}$ and take (17) into (13), we have
\begin{equation}
\begin{aligned}
&\|M^{n+m+1}-M^{n+1}\|_{B_{2,\infty}^{-\frac{1}{2}}}\le \|S_{n+m+1}M_0-S_{n+1}M_0\|_{B_{2,\infty}^{-\frac{1}{2}}}\\
&\ \ \ +C\int_0^t\|(M^{n+m}-M^n)(\tau)\|_{B_{2,\infty}^{-\frac{1}{2}}}\left(1-\log\left(\|(M^{n+m}-M^n)(\tau)\|_{B_{2,\infty}^{-\frac{1}{2}}}\right)\right)d\tau.
\end{aligned}
\end{equation}
Denote $E_{n,m}(t)=\|M^{n+m}-M^n\|_{B_{2,\infty}^{-\frac{1}{2}}}$, then (18) can be rewritten as
\begin{equation}
E_{n+1,m}(t) \le C\Big(E_{n+1,m}(0)+\int_0^t E_{n,m}\left(1-\log E_{n,m}\right)(\tau)d\tau\Big).
\end{equation}

Since for any function h, $S_{n+m+1}h(0)-S_{n+1}h(0)=\displaystyle\sum_{j=n+1}^{n+m}\Delta_jh(0)$, there exits a constant $C$ which only depend on T, such that [28]
\begin{equation}
\begin{aligned}
&\|M_0^{m+n+1}-M_0^{n+1}\|_{B_{2,\infty}^{-\frac{1}{2}}}=\|S_{m+n+1}M_0-S_{n+1}M_0\|_{B_{2,\infty}^{-\frac{1}{2}}}\\
&=\|\sum_{k=n+1}^{m+n}\Delta_kM_0\|_{B_{2,\infty}^{-\frac{1}{2}}} \le c2^{-n}.
\end{aligned}
\end{equation}
Arguing by induction, for a fixed number $m \in \mathbb{N}$, we get
\begin{equation}
E_{n+1,m}(t)\left(1-\log E_{n+1,m}\right) \le C\displaystyle\sum_{k=0}^n 2^{-(n-k)}\frac{(TC)^k}{k!}+\frac{(TC)^{n+1}}{(n+1)!}E_{0,m}(t)\left(1-\log E_{0,m}\right).
\end{equation}
Due to $\|M\|_{B_{2,\infty}^{-\frac{1}{2}}}$ and $C$ are bounded, there exists a constant $C_T$ independent of $m,n$ such that
\begin{equation}
\|M^{n+m+1}-M^{n+1}\|_{B_{2,\infty}^{-\frac{1}{2}}} \le C_T2^{-n}.
\end{equation}
Thus $(M^n)_{n\in\mathbb{N}}$ is a Cauchy sequence in $C([0,T);B_{2,\infty}^{-\frac{1}{2}})^{2n}$. Because $\|M^n\|_{B_{2,1}^{\frac{1}{2}}} \le M$, and take advantage of the interpolation inequality (14), we deduce that $(M^n)$ tends to $M$ in $C([0,T);B_{2,1}^{-\frac{1}{2}})^{2n}$.

\textbf{Step 4.}
Finally, we prove the uniqueness of the solution to (4). Suppose $M_1$ and $M_2$ are two solutions of (4), we have
\begin{equation}
(M_2-M_1)_t+a_2(M_2-M_1)_x=-(a_2-a_1)M_{1x}+B_2(M_2-M_1)+(B_2-B_1)M_1.
\end{equation}
Similarly with Step 3, by virtual of Lemma 2.6, we have the following estimation
\begin{equation}
\|M_2-M_1\|_{B_{2,\infty}^{-\frac{1}{2}}} \le e^{c\int_0^ta_2(\tau)d\tau}\|M_{20}-M_{10}\|_{B_{2,\infty}^{-\frac{1}{2}}} +\int_0^te^{c\int_\tau^ta_2(s)ds}\|M_2-M_1\|_{B_{2,1}^{-\frac{1}{2}}}d\tau.
\end{equation}
Denote $\mathcal{M}(t)=e^{-c\int_0^ta_2(\tau)d\tau}\|M_2-M_1\|_{B_{2,\infty}^{-\frac{1}{2}}}$, for a sufficient large enough constant $c$, we have $\mathcal{M}(t) < 1$. Thanks to interpolation inequality (14), we have
\begin{equation}
\begin{aligned}
\mathcal{M}(t)&\le \mathcal{M}(0)+\int_0^t\mathcal{M}(\tau)\log\left(e+\frac{\|M_2-M_1\|_{B_{2,\infty}^{\frac{1}{2}}}}{\|M_2-M_1\|_{B_{2,\infty}^{-\frac{1}{2}}}}\right)d\tau\\
&\le \mathcal{M}(0)+c\int_0^t\mathcal{M}(\tau)\left(1-\log(\mathcal{M}(\tau))\right)d\tau.
\end{aligned}
\end{equation}
By virtual of Osgood Lemma 2.5, we take $\mu(r)=r(1-\log r)$, $\mathcal{N}(r)=\log(1-\log(r))$. Finally, we obtain
\begin{equation}
\mathcal{M}(t) \le e^{1-[1-\log \mathcal{M}(0)]e^{-ct}}.
\end{equation}
With the interpolation inequality (14), we complete the proof of uniqueness.

\end{proof}

\section{Blow-up criteria}

In this section, we present a blow-up criterion for (1) in Sobolev space. The first result is as follows.
\begin{theorem}
Let $M_0 \in (H^s(\mathbb{R}))^{2n}$ with $s>\tfrac{1}{2}$ , and $T>0$ be he maximal existence time of the solution $M(t)$ to (4). If the solution blows up in finite time, then we have
$$\int_0^T\|M(\tau)\|_{L^\infty}^2d\tau = \infty.$$
\end{theorem}

\begin{proof}
\textbf{Step 1.}
For $s \in (\frac{1}{2},1]$, applying Lemma 2.6 to (4), we have

\begin{equation}
\|M(t)\|_{H^s} \le \|M_0\|_{H^s} +C\int_0^t\|a(\tau)\|_{W^{1,\infty}}\|M(\tau)\|_{H^s}+\|B(\tau)M(\tau)\|_{H^s}d\tau,
\end{equation}

Because of the definition of $\|a\|_{W^{1,\infty}}=\displaystyle\sup_{x\in\mathbb{R}}(\|a\|_{L^\infty}+\|a_x\|_{L^\infty})$, we claim that
\begin{equation}
\|a\|_{W^{1,\infty}} \le C\|M\|_{L^\infty}^2,
\end{equation}
here we used the Young inequality
\begin{equation*}
\|U\|_{L^r}=\|G*M\|_{L^r}\le \|G\|_{L^p}\|M\|_{L^q},\ \ \ \frac{1}{r}+1=\frac{1}{p}+\frac{1}{q},
\end{equation*}
where $G=\frac{1}{2}e^{-|x|}$ is Poisson kernel of Bessel potential operator $(1-\partial_x^2)$. For the second term in the integral in (27), by virtual of Lemma 2.4, we have

\begin{equation}
\|BM\|_{H^s}\le C(\|B\|_{L^\infty}\|M\|_{H^s}+\|B\|_{H^s}\|M\|_{L^\infty}) \le C\|M\|_{L^\infty}^2\|M\|_{H^s}.
\end{equation}

Take (28) and (29) into(27), it's easy to deduce that
\begin{equation}
\|M(t)\|_{H^s} \le \|M_0\|_{H^s}+C\int_0^t\|M(\tau)\|_{L^\infty}^2\|M(\tau)\|_{H^s}d\tau,
\end{equation}
taking advantage of Gronwall's inequality, we obtain
\begin{equation}
\|M(t)\|_{H^s} \le \|M_0\|_{H^s}\cdot \exp\Bigg\{C\int_0^t\|M(\tau)\|_{L^\infty}^2d\tau\Bigg\}.
\end{equation}

Therefore, if $T < \infty$ satisfies $\displaystyle\int_0^t\|M(\tau)\|_{L^\infty}^2d\tau < \infty$, then we deduce that
\begin{equation}
\limsup_{t\rightarrow T}\|M(t)\|_{H^s} < \infty,
\end{equation}
which contradicts the assumption that $T$ is the maximal existence time.

\textbf{Step 2.}
For $s \in (1,2]$, by differentiating (4) with respect to $x$, we have

\begin{equation}
M_{xt}+a(U,U_x)M_{xx}=-a_x(U,U_x)M_x+[B(U,U_x)M]_x,
\end{equation}

By virtue of Lemma 2.7 with $s-1 \in (0,1]$ 
, we obtain
\begin{equation}
\begin{aligned}
\|M_x(t)\|_{H^{s-1}} &\le \|M_{0,x}\|_{H^{s-1}} +C\int_0^t\|a(\tau)\|_{W^{1,\infty}}\|M_x(\tau)\|_{H^{s-1}}\\
&\ \ \ +\|a(\tau)_xM(\tau)_x\|_{H^{s-1}}d\tau+\|(B(\tau)M(\tau))_x\|_{H^{s-1}},
\end{aligned}
\end{equation}

For $\|a_xM_x\|_{H^{s-1}}$, applying Lemma 2.4, and using the fact that $\|f_x\|_{H^{s-1}} \le C\|f\|_{H^s}$, we have
\begin{equation}
\begin{aligned}
\|a_xM_x\|_{H^{s-1}} &\le C(\|a_x\|_{H^s}\|M\|_{L^\infty}+\|a_x\|_{L^\infty}\|M_x\|_{H^{s-1}})\\
&\le C(\|a\|_{H^{s+1}}\|M\|_{L^\infty}+\|a_x\|_{L^\infty}\|M\|_{H^s})\\
&\le C\|M\|_{L^\infty}^2\|M\|_{H^s}.
\end{aligned}
\end{equation}

It is easy to check that just along with (30) for $s-1$ instead of s ensures that (31) is still holds.

\textbf{Step 3.} For $s > 2$, applying $\Lambda^s$ to (4), we have

\begin{equation}
(\Lambda^sM)_t+\Lambda^s(aM_x)=\Lambda^s(BM).
\end{equation}

Multiplying by $\Lambda^sM$ and integrating over $R$, we have

\begin{equation}
\begin{aligned}
&\int_R(\Lambda^sM)(\Lambda^sM_t)dx+\int_R(\Lambda^sM)(\Lambda^sM)_xa dx\\
&\ \ \ =\int_R -\Big[\Lambda^s(M_xa)-(\Lambda^sM_x)a\Big](\Lambda^sM) +[\Lambda^s(BM)](\Lambda^sM)dx.
\end{aligned}
\end{equation}

Using H\"older's inequality and integrating by parts, we obtain

\begin{equation}
\begin{aligned}
\|M(t)\|_{H^s} &\le \|M_0\|_{H^s}+C\displaystyle\int_0^t\|a(\tau)\|_{W^{1,\infty}}\|M(\tau)\|_{H^s}\\
&+\|\Lambda^s(M_x(\tau)a(\tau))-(\Lambda^sM_x(\tau))a(\tau)\|_{L^2}+ \|B(\tau)M(\tau)\|_{H^s}d\tau.
\end{aligned}
\end{equation}

Thanks to Lemma 2.8 and Sobolev embedding $H^{\frac{3}{2}+\epsilon} \hookrightarrow W^{1,\infty}$ with $\epsilon >0$, we infer that

\begin{equation*}
\begin{aligned}
\|\Lambda^s(M_xa)-(\Lambda^sM_x)a\|_{L^2} &\le C\Big(\|\Lambda^sa\|_{L^2(R)}\|M_x\|_{L^\infty(R)}+\|a_x\|_{L^\infty(R)}\|\Lambda^{s-1}M_x\|_{L^2(R)}\Big)\\
&\le C\Big(\|a\|_{H^s}\|M\|_{H^{\frac{3}{2}+\epsilon}}+\|a\|_{W^{1,\infty}}\|M\|_{H^s}\Big).
\end{aligned}
\end{equation*}

As we have proved in Step 2 that if (31) holds for $s \in (1,2]$, then $\|M\|_{H^{\frac{3}{2}+\epsilon}}$ is uniformly bounded for $\frac{3}{2}+\epsilon \in (1,2)$, we note that (31) is still holds in step 3, which is contradict to the assumption that $T$ is the maximum exist time. Thus we complete the proof of the theorem for $s > 2$. Consequently, we prove the theorem by Step 1-3.
\end{proof}

By Sobolev's embedding theorem and Theorem 4.1, we have the following corollary.
\begin{corollary}
Let $M_0 \in (H^s(\mathbb{R}))^{2n}$ with $s>\tfrac{1}{2}$ , and $T>0$ be he maximal existence time of the solution $M(t)$ to (4). Then the solution blows up in finite time if and only if
$$\limsup_{t\rightarrow T}\|m_i(t)\|_{L^\infty}=+\infty\ \ \ or\ \ \ \limsup_{t\rightarrow T}\|n_i(t)\|_{L^\infty}=+\infty,\ \ \ i=\{1\cdots N\}.$$
\end{corollary}
\begin{proof}
In fact, in Theorem 4.1 we have already deduced that if $M(t)$ blows up in finite time, we have $\int_0^T\|M(\tau)\|_{L^\infty}^2d\tau = \infty$ throw (31), which means $\displaystyle\limsup_{t\rightarrow T}\|m_i(t)\|=\infty$ or $\displaystyle\limsup_{t\rightarrow T}\|n_i(t)\|=\infty$. On the other hand, by Sobolev's embedding theorem $\|M(t)\|_{L^\infty} \le C\|M(t)\|_{H^s}$ with $s>\frac{1}{2}$, we have $\displaystyle \limsup_{t \rightarrow T}\|M(t)\|_{H^s}=\infty$.
\end{proof}
Next, we will discuss the second blow-up criterion for some reduced cases.

\textbf{Case 1.} $N=2, m_1=n_2=m, m_2=n_1=n, u_1=v_2=u, u_2=v_1=v, $
\begin{equation}
\begin{cases}
m_t+2uvm_x=-4u_xvm-2uv_xm,\\
n_t+2uvn_x=-4uv_xn-2u_xvn.\\
\end{cases}
\end{equation}

Consider the following initial value problem:

\begin{equation}
\begin{cases}
\dfrac{d\Phi(t,x)}{dt}=2uv(t,\Phi(t,x)),\\
\Phi(0,x)=x,
\end{cases}
\end{equation}
for the flow generated by $2uv$. Applying classical results in the theory of ODEs, we have the following lemma.

\begin{lemma}
Let $(m_0,n_0) \in (H^s(\mathbb{R}))^2$ with $s>\frac{1}{2}$ and let $T>0$ be the maximal existence time of the corresponding solution $(m,n)$ to (39). Then (40) has a unique solution $\Phi \in C^1([0,T];C^1(\mathbb{R}))$. Moreover, the map $\Phi(t,\cdot)$ is an increasing diffeomorphism of $\mathbb{R}$ with

\begin{equation}
\Phi_x(t,x)=\exp\left(\int_0^t2(u_xv+uv_x)(s,\Phi(s,x))ds\right)>0.
\end{equation}
\end{lemma}

Now we show the blow up condition for (39).

\begin{theorem}
Let $(m_0,n_0) \in (H^s(\mathbb{R}))^2$ with $s>\frac{1}{2}$ and $T$ be the maximal existence time of the solution $(m,n)$ to (39). Then the solution blows up in finite time if and only if
\begin{equation}
\liminf_{t \rightarrow T}\{\inf_{x\in\mathbb{R}}u_xv(t,x)\} = -\infty\ \  or\ \  \liminf_{t \rightarrow T}\{\inf_{x\in\mathbb{R}}uv_x(t,x)\} = -\infty. \ \
\end{equation}
\end{theorem}

\begin{proof}
Consider (39) along the flow $2uv$, we have

\begin{equation*}
\begin{aligned}
\dfrac{dm}{dt}(t,\Phi(t,x))&=\Big(m_t+m_x\Phi_t)(t,\Phi(t,x)\Big)=\Big(m_t+2uvm_x\Big)(t,\Phi(t,x))\\
&=\Big(-(4u_xv+2uv_x)m(t,\Phi(t,x))\Big),
\end{aligned}
\end{equation*}

\begin{equation*}
\begin{aligned}
\dfrac{dn}{dt}(t,\Phi(t,x))&=\Big(n_t+n_x\Phi_t)(t,\Phi(t,x)\Big)=\Big(n_t+2uvn_x\Big)(t,\Phi(t,x))\\
&=\Big(-(4uv_x+2u_xv)n(t,\Phi(t,x))\Big).
\end{aligned}
\end{equation*}

Integrating with respect to $t$, we get

\begin{equation}
m(t,\Phi(t,x))=m_0\cdot\exp\Big(\int_0^t-(4u_xv+2uv_x)(s,\Phi(s,x))ds\Big),
\end{equation}
\begin{equation}
n(t,\Phi(t,x))=n_0\cdot\exp\Big(\int_0^t-(4uv_x+2u_xv)(s,\Phi(s,x))ds\Big).
\end{equation}

Suppose that $u_xv>-C$ and $uv_x>-C$, we obtain the estimation for $\|m(t,x)\|_{L^\infty}$ and $\|n(t,x)\|_{L^\infty}$
\begin{equation}
\begin{aligned}
&\|m(t,\cdot)\|_{L^\infty} \le \|m_0\|_{L^\infty}e^{6CT} \le C\|m_0\|_{H^s}e^{6CT} < \infty,\\
&\|n(t,\cdot)\|_{L^\infty}\  \le \|n_0\|_{L^\infty}e^{6CT}\  \le C\|n_0\|_{H^s}e^{6CT} < \infty.
\end{aligned}
\end{equation}
If the maximal existence time $T$ is finite, then $\|m(t)\|_{L^\infty}$ and $\|n(t)\|_{L^\infty}$ are both bounded,  which contradicts to Corollary 4.2.
On the other hand, as $u=G*m$ and $v=G*n$, with Young inequality, we have
\begin{equation}
\begin{aligned}
\|u\|_{L^\infty},\|u_x\|_{L^\infty} &\le \|m\|_{L^\infty},\\
\|v\|_{L^\infty},\|v_x\|_{L^\infty} &\le \|n\|_{L^\infty}.
\end{aligned}
\end{equation}
Thus, if (42) holds, the $L^\infty$ norm of $m(t)$ or $n(t)$ will tends to $+\infty$ as $t \rightarrow T$. By Corollary 4.2 and Sobolev embedding theorem, the solution $(m,n)$ of (39) will blow up in finite time.
\end{proof}

\textbf{Case 2.} $N=2, m_1=n_1=m, m_2=n_2=n, u_1=v_1=u, u_2=v_2=v, $
\begin{equation}
\begin{cases}
m_t+(u^2+v^2)m_x=-3m(uu_x+vv_x)-n(u_xv-uv_x),\\
n_t+(u^2+v^2)n_x=-3n(uu_x+vv_x)-m(uv_x-u_xv).\\
\end{cases}
\end{equation}

Consider the following initial value problem:

\begin{equation}
\begin{cases}
\dfrac{d\Phi(t,x)}{dt}=(u^2+v^2)(t,\Phi(t,x)),\\
\Phi(0,x)=x,
\end{cases}
\end{equation}
for the flow generated by $u^2+v^2$. Similarly with Lemma 4.3, we have the following lemma.

\begin{lemma}
Let $(m_0,n_0) \in (H^s(\mathbb{R}))^2$ with $s>\frac{1}{2}$ and let $T>0$ be the maximal existence time of the corresponding solution $(m,n)$ to (47). Then (48) has a unique solution $\Phi \in C^1([0,T];C^1(\mathbb{R}))$. Moreover, the map $\Phi(t,\cdot)$ is an increasing diffeomorphism of $\mathbb{R}$ with

\begin{equation}
\Phi_x(t,x)=\exp\left(\int_0^t2(uu_x+vv_x)(s,\Phi(s,x))ds\right)>0.
\end{equation}
\end{lemma}

The blow up condition for (47) is as follows.

\begin{theorem}
Let $(m_0,n_0) \in (H^s(\mathbb{R}))^2$ with $s>\frac{1}{2}$ and $T$ be the maximal existence time of the solution $(m,t)$ to (47). Then the solution blows up in finite time if and only if
\begin{equation}
\liminf_{t \rightarrow T}\{\inf_{x\in\mathbb{R}}(uu_x+vv_x)(t,x)\} = -\infty\ \  or\ \  \limsup_{t \rightarrow T}\|(uv_x-uv_x)(t)\|_{L^\infty} = +\infty. \ \
\end{equation}
\end{theorem}

\begin{proof}
With a similar technique in proof of Theorem 4.4, we obtain

\begin{equation*}
\begin{aligned}
\dfrac{dm}{dt}(t,\Phi(t,x))&=\Big(m_t+m_x\Phi_t)(t,\Phi(t,x)\Big)=\Big(m_t+(u^2+v^2)m_x\Big)(t,\Phi(t,x))\\
&=\Big(-3m(uu_x+vv_x)-n(u_xv-uv_x)\Big)(t,\Phi(t,x)),
\end{aligned}
\end{equation*}

\begin{equation*}
\begin{aligned}
\dfrac{dn}{dt}(t,\Phi(t,x))&=\Big(n_t+n_x\Phi_t)(t,\Phi(t,x)\Big)=\Big(n_t+(u^2+v^2)n_x\Big)(t,\Phi(t,x))\\
&=\Big(-3n(uu_x+vv_x)-m(uv_x-u_xv)\Big)(t,\Phi(t,x)).
\end{aligned}
\end{equation*}

Integrating with respect to $t$, we have

\begin{equation}
\begin{aligned}
m(t,\Phi(t,x))&=m_0\cdot\exp\Big(\int_0^t-3(uu_x+vv_x)(s,\Phi(s,x))ds\Big)\\
&\ \ \ -\int_0^tn(u_xv-uv_x)\exp\Big(\int_s^t-3(uu_x+vv_x)(\tau,\Phi(\tau,x))d\tau\Big)(s,\Phi(s,x))ds,
\end{aligned}
\end{equation}

\begin{equation}
\begin{aligned}
n(t,\Phi(t,x))&=n_0\cdot\exp\Big(\int_0^t-3(uu_x+vv_x)(s,\Phi(s,x))ds\Big)\\
&\ \ \ -\int_0^tm(uv_x-u_xv)\exp\Big(\int_s^t-3(uu_x+vv_x)(\tau,\Phi(\tau,x))d\tau\Big)(s,\Phi(s,x))ds.
\end{aligned}
\end{equation}

Suppose that $uu_x+vv_x>-C$ and $\|u_xv-uv_x\|_{L^\infty}<C$, we get the boundary of $\|m(t,x)\|_{L^\infty}+\|n(t,x)\|_{L^\infty}$
\begin{equation}
\|m(t,\cdot)\|_{L^\infty}+\|n(t,\cdot)\|_{L^\infty} \le  e^{3CT}\Big((\|m_0\|_{L^\infty}+\|n_0\|_{L^\infty})+C\int_0^t(\|m(s,\cdot)\|_{L^\infty}+\|n(s,\cdot)\|_{L^\infty})ds\Big).
\end{equation}

Applying Gronwall's inequality, we obtain

\begin{equation}
\begin{aligned}
\|m(t,\cdot)\|_{L^\infty}+\|n(t,\cdot)\|_{L^\infty} &\le (\|m_0\|_{L^\infty}+\|n_0\|_{L^\infty})e^{3CT}\cdot e^{CTe^{3CT}}\\
&\le C(\|m_0\|_{H^s}+\|n_0\|_{H^s})e^{3CT}\cdot e^{CTe^{3CT}}<\infty.
\end{aligned}
\end{equation}
If the maximal existence time $T$ is finite, then $\|m(t)\|_{L^\infty}$ and $\|n(t)\|_{L^\infty}$ are both bounded,  which contradicts to Corollary 4.2.
On the other hand, with a similar argument in Theorem 4.4, if (50) holds, by (46) we obtain the solution $(m,n)$ of (47) will blow up in finite time.
\end{proof}

\textbf{Case 3.} For $N=1$, (1) is reduced to GX equation, we claim that Theorem 4.4 is also suitable for (2), the proof of the following theorem is very similar to Theomre 4.4 just with a slight modification, thus we only show the result here.
\begin{theorem}
Let $(m_0,n_0) \in (H^s(\mathbb{R}))^2$ with $s>\frac{1}{2}$ and $T$ be the maximal existence time of the solution $(m,n)$ to (2). Then the solution blows up in finite time if and only if (42) holds
\begin{equation*}
\liminf_{t \rightarrow T}\{\inf_{x\in\mathbb{R}}u_xv(t,x)\} = -\infty\ \  or\ \  \liminf_{t \rightarrow T}\{\inf_{x\in\mathbb{R}}uv_x(t,x)\} = -\infty. \ \
\end{equation*}
\end{theorem}

\section{Global existence}
In this section, we will discuss the global existence of (2), which is a two-component reduction of (1).
First we show that GX equation admits the following conserved quantity
\begin{equation}
H=\int_\mathbb{R}mvdx=\int_\mathbb{R}nudx=\int_\mathbb{R}(uv+u_xv_x)dx,
\end{equation}
which is crucial to the proof of global existence.
\begin{lemma}
Let $(m_0,n_0) \in (H^s(\mathbb{R}))^2$ with $s>\frac{1}{2}$, $T>0$ is the maximum existence time of the corresponding solution $(m,n)$ of (2). Then for all $t \in [0,T)$, we have
$$\frac{d}{dt}H=0.$$
\end{lemma}

\begin{proof}
Arguing by density, we only consider the case for $(u,v) \in C_0^\infty(\mathbb{R})$. Taking advantage of (2) and integration by parts, we deduce that
\begin{equation}
\begin{aligned}
\frac{d}{dt}\int_\mathbb{R}(uv+u_xv_x)dx&=\int_\mathbb{R}(u_tv+uv_t+u_{tx}v_x+u_xv_{tx})dx=\int_\mathbb{R}(m_tv+n_tu)dx\\
&=-\int_\mathbb{R}(uv^2m_x+3u_xv^2m+u^2vn_x+3u^2v_xn)dx\\
&=2\int_\mathbb{R}(uvv_xm-u_xv^2m+uu_xvn-u^2v_xn)dx\\
&=2\int_\mathbb{R}(uv_x-u_xv)(vm-un)dx\\
&=2\int_\mathbb{R}(uv_x-u_xv)(uv_{xx}-u_{xx}v)dx=\int[(uv_x-u_xv)^2]_xdx=0.
\end{aligned}
\end{equation}
\end{proof}

The global existence of GX equation can be stated as follows.
\begin{theorem}
Suppose that $(m_0,n_0)\in (H^s(\mathbb{R}))^2$ with $s >\frac{1}{2}$, and $m_0,n_0$ do not change sign. Then the corresponding strong solution $(m,n)$ of (2) is global in time.
\end{theorem}

\begin{proof}
We only consider that $m_0,n_0 \ge0$ and others are similar. Consider (2) along the curve $uv$, with a similar computation in Theorem 4.4, we have
\begin{equation}
m(t,\Phi(t,x))=m_0\cdot\exp\Big(\int_0^t-3u_xv(s,\Phi(s,x))ds\Big),
\end{equation}
\begin{equation}
n(t,\Phi(t,x))=n_0\cdot\exp\Big(\int_0^t-3uv_x(s,\Phi(s,x))ds\Big).
\end{equation}
It is obvious that $m,n$ will not change sign if $m_0,n_0$ are nonnegative.

As $m=u-u_{xx}$, we have $u=G*m=\int_\mathbb{R} \frac{1}{2}e^{-|x-y|}m(y)dy$, where $G(x)=\frac{1}{2}e^{-|x|}$ is the poisson kernel of Bessel potential operator. It is obvious that $u>0$ as
\begin{equation}
u=\displaystyle \frac{1}{2}e^x\int_x^\infty e^{-y}m(y)dy+\frac{1}{2}e^{-x}\int_{-\infty}^xe^ym(y)dy >0.
\end{equation}
For $u_x$, we have
\begin{equation}
u_x=\displaystyle \frac{1}{2}e^x\int_x^\infty e^{-y}m(y)dy-\frac{1}{2}e^{-x}\int_{-\infty}^xe^ym(y)dy.
\end{equation}
Plug (59) and (60), we can check $u+u_x \ge 0$ and $u-u_x \ge 0$ easily, that is to say
\begin{equation}
(1 \pm \partial_x)u(t,x) \ge 0,\ \ \ \forall (t,x) \in [0,T) \times \mathbb{R},
\end{equation}
here $T>0$ is the maximal existence time of (2). With a similar computation, we also have
\begin{equation}
(1 \pm \partial_x)v(t,x) \ge 0,\ \ \ \forall (t,x) \in [0,T) \times \mathbb{R}.
\end{equation}

In order to prove the global existence, a natural idea is to consider whether the $\|\cdot\|_{L^\infty}$ of $u$, $u_x$, $v$ and $v_x$ can be controlled by time-parameter $t$ through Theorem 4.7. By Sobolev's embedding theorem $H^1 \hookrightarrow L^\infty$ and (61)-(62), it remains to prove that $\|u\|_{H^1}$ and $\|v\|_{H^1}$ are bounded.

Differentiating $\|u\|_{H^1}^2$ with respect to $t$ and integrating by parts, we have
\begin{equation}
\begin{aligned}
\dfrac{d}{dt}\|u\|_{H^1}^2&=2\int_\mathbb{R} uu_t+u_xu_{xt}dx=2\int_\mathbb{R} um_tdx\\
&=\int_\mathbb{R} -2u^2vm_x-6uu_xvmdx=\int 2u^2v_xm-2uu_xmvdx\\
&\le 2\|u\|_{L^\infty}^2\int_\mathbb{R} |mv_x|dx+(\|u\|_{L^\infty}^2+\|u_x\|_{L^\infty}^2)\int_\mathbb{R} mvdx\\
&\le 4\|u\|_{L^\infty}^2\int_\mathbb{R} mvdx =4\|u\|_{L^\infty}^2\int_\mathbb{R} m_0v_0dx\\
&\le 4C\|u\|_{H^1}^2\|m_0\|_{L^2}\|n_0\|_{L^2}.
\end{aligned}
\end{equation}

Applying the Gronwall's inequality, we have
$$\|u_x\|_{L^\infty} \le \|u\|_{L^\infty} \le C\|u\|_{H^1} \le C\|u_0\|_{H^1}e^{2\|m_0\|_{L^2}\|n_0\|_{L^2}Ct},$$
and similarly, we also deduce that
$$\|v_x\|_{L^\infty} \le \|v\|_{L^\infty} \le C\|v\|_{H^1} \le C\|v_0\|_{H^1}e^{2\|m_0\|_{L^2}\|n_0\|_{L^2}Ct}.$$
Thus, by Theorem 4.7, the proof is completed.
\end{proof}

\begin{remark}
Assume $(m_0,n_0) \in (H^s(\mathbb{R}))^2$ with $s>\frac{1}{2}$ are the initial data of (39). By (43) and (44), we know if $m_0,n_0 \ge 0$, then the corresponding solution $(m,n)$ will remains non-negative for all $t \in [0,T)$. It is easily to check that both
$$H_1=\int_\mathbb{R}(u^2+u_x^2)dx,\ \ \ H_2=\int_\mathbb{R}(v^2+v_x^2)dx,$$
are conserved quantities of (39). Thus, with a similar computation of Theorem 5.2, we claim that (39) is global in time under the initial condition $(m_0 \ge 0,n_0 \ge 0) \in (H^s(\mathbb{R}))^2$ with $s>\frac{1}{2}$.
\end{remark}

\section{Peakon and periodic Peakon}
In this section, we show that system (1) also admits peakons and periodic peakons, just as same as Camassa-Holm equation. The peaked and periodic peaked solutions (with periodic 1) of CH are stated as follows
$$\psi_c=ce^{-|x-ct|}, \ \ \ \ \ \ \psi_p=\dfrac{1}{ch(\tfrac{1}{2})}ch(\zeta),$$
where $\zeta=\frac{1}{2}-(x-ct)+[x-ct]$ and $c$ is an arbitrary constant.

Obviously, the function $\psi_c$ and $\psi_p$ are not smoothness, thus we should order the definition of weak solution to system (1) first.
\begin{definition} Solutions $u_i$ and $v_i$ of system (1) are called weak, if for any function $\phi(t,x) \in C^1([0,T];C_0^\infty({\Omega}))$, we have
\begin{equation}
\begin{aligned}
\displaystyle \int_0^T\int_\Omega\Big(m_i\phi_t+\sum_{j=1}^N&\left(-2m_iu_{jx}v_j-m_iu_jv_{jx}-m_{ix}u_jv_j-u_{ix}m_jv_j+u_im_jv_{jx}\right)\Big)\phi dxdt\\
&=\displaystyle \int_\Omega m_i(T)\phi(T)dx - \int_\Omega m_i(0)\phi(0) dx,\\
\displaystyle\int_0^T \int_\Omega\Big(n_i\phi_t+\sum_{j=1}^N&\left(-2n_iu_jv_{jx}-n_iu_{jx}v_j-n_{ix}u_jv_j-v_{ix}n_ju_j+v_in_ju_{jx}\right)\Big)\phi dxdt\\
&=\displaystyle \int_\Omega n_i(T)\phi(T)dx - \int_\Omega n_i(0)\phi(0) dx.\\
\end{aligned}
\end{equation}
\end{definition}

\textbf{Peakon}
Now, we suppose $u_i=p_i(t)E(\xi)$, $v_i=q_i(t)E(\xi)$, where $E=e^{-|\xi|}=e^{-|x-s(t)|}$, $i=\{1,\cdots,N\}$. Put $u_i$ and $v_i$ into (1) and multiply both sides with $\phi$, take integral with respect to $x$, we have
\begin{equation}
\begin{aligned}
I_1=\displaystyle\int_\mathbb{R}m_{it}\phi dx &=\sum_{j=1}^N\int_\mathbb{R}\left(-2m_iu_{jx}v_j-m_iu_jv_{jx}-m_{ix}u_jv_j-u_{ix}m_jv_j+u_im_jv_{jx}\right)\phi dx\\
&=p_i\sum_{j=1}^Np_jq_j\int_\mathbb{R}\left(-4E^2E_x\phi+3EE_xE_{xx}\phi+E^2E_{xxx}\phi \right)dx\\
&=p_i\sum_{j=1}^Np_jq_j\int_\mathbb{R}\left(-4E^2E_x\phi-\frac{1}{2}E_x^3\phi+\frac{3}{2}EE_x^2\phi_x+E^2E_x\phi_{xx}\right)dx\\
&=p_i\sum_{j=1}^Np_jq_jI_2.
\end{aligned}
\end{equation}
For the left hand side of (65) $I_1$, we have
\begin{equation}
\begin{aligned}
I_1&=\int_\mathbb{R}(p_i(E-E_{xx}))_t\phi dx\\
&=\int_{-\infty}^s(\dot p_i-p_i\dot s)e^{x-s}(\phi-\phi_{xx})dx+\int_s^\infty(\dot p_i+p_i\dot s)e^{s-x}(\phi-\phi_{xx})dx\\
&=(\dot p_i-p_i\dot s)(e^{x-s}(\phi-\phi_x)|_{-\infty}^s)-(\dot p_i+p_i\dot s)(e^{s-x}(\phi+\phi_x)|_s^\infty)\\
&=2\dot p_i\phi(s)+2p_i\dot s\phi_x(s),
\end{aligned}
\end{equation}
and for $I_2$, with a simple computation, we obtain
\begin{equation}
\begin{aligned}
I_2&=\displaystyle\int_{-\infty}^se^{3(x-s)}(-\frac{9}{2}\phi+\frac{3}{2}\phi_x+\phi_{xx})dx+\int_s^\infty e^{3(s-x)}(\frac{9}{2}\phi+\frac{3}{2}\phi_x-\phi_{xx})\\
&=\int_{-\infty}^s-\frac{3}{2}(e^{3(x-s)}\phi)_x+(e^{3(x-s)}\phi_x)_xdx+\int_s^\infty-\frac{3}{2}(e^{3(s-x)}\phi)_x-(e^{3(s-x)}\phi_x)_xdx\\
&=2\phi_x(s).
\end{aligned}
\end{equation}

Take (66) and (67) into (65) and compare the coefficients of $\phi(s)$ and $\phi_x(s)$, we have the following ODES,
\begin{equation}
\begin{cases}
\dot p_i=0,\\
p_i\dot s=p_i\sum_{j=1}^Np_jq_j.
\end{cases}
\end{equation}

With a similar computation, it's easy to deduce that
\begin{equation}
\begin{cases}
\dot q_i=0,\\
q_i\dot s=q_i\sum_{j=1}^Np_jq_j.
\end{cases}
\end{equation}

Thus, Eq(1) admits peak solution
\begin{equation}
u_i=p_ie^{-|x-ct|}, v_i=q_ie^{-|x-ct|},
\end{equation}
where $p_i$ and $q_i$ are some constants, and the velocity $c=\sum_{j=1}^Np_jq_j$. It's easy to check that (70) satisfies (64) just with an integration by parts with respect to $t$, here we omit it.

\textbf{Periodic peakon}

Similarly as (65), suppose $u_i=p_i(t)E,v_i=q_i(t)E,E=ch(x-s_i(t)-[x-s(t)]-\frac{1}{2})$, $x \in \mathbb{S}=\mathbb{R}/\mathbb{Z}$, $i=\{1,\cdots, N\}$, where $ch(x)$ is hyperbolic cosine(cosh). With a text function $\phi \in C^\infty([0,T);\Omega)$ and a slightly modification in (65), we have
\begin{equation}
\begin{aligned}
J_1=\displaystyle\int_\mathbb{S}m_{it}\phi dx&=p_i\sum_{j=1}^Np_jq_j\int_\mathbb{S}\left(-4E^2E_x\phi-\frac{1}{2}E_x^3\phi+\frac{3}{2}EE_x^2\phi_x+E^2E_x\phi_{xx}\right)dx\\
&=p_i\sum_{j=1}^Np_jq_jJ_2.
\end{aligned}
\end{equation}
Denote $\xi=x-s+\tfrac{1}{2}$, $\zeta=x-s-\tfrac{1}{2}$, with a direct computation for $J_1$ and $J_2$, we have
\begin{equation}
\begin{aligned}
J_1&=\int_\mathbb{S}(p_i(E-E_{xx}))_t\phi dx\\
&=\int_0^s(\dot p_i\cdot ch\xi-p_i\dot s \cdot sh\xi)(\phi-\phi_{xx})dx+\int_s^1(\dot p_i ch\zeta-p_i\dot s \cdot sh\zeta)(\phi-\phi_{xx})dx\\
&=\dot p_i\Big(sh\tfrac{1}{2}\phi(s)-ch\tfrac{1}{2}\phi_x(s)\Big)-p_i\dot s\Big(ch\tfrac{1}{2}\phi(s)-sh\tfrac{1}{2}\phi_x(s)\Big)\\
&+\dot p_i\Big(sh\tfrac{1}{2}\phi(s)+ch\tfrac{1}{2}\phi_x(s)\Big)-p_i\dot s\Big(-ch\tfrac{1}{2}\phi(s)-sh\tfrac{1}{2}\phi_x(s)\Big)\\
&=2\dot p_ish\tfrac{1}{2}\phi(s)+2p_i\dot s sh\tfrac{1}{2}\phi_x(s),
\end{aligned}
\end{equation}

\begin{equation}
\begin{aligned}
J_2&=\displaystyle\int_0^s\left(-4ch^2\xi\cdot sh\xi\cdot\phi-\frac{1}{2}sh^3\xi\cdot\phi+\frac{3}{2}ch\xi\cdot sh^2\xi\cdot\phi_x+ch^2\xi\cdot sh\xi\cdot\phi_{xx}\right)\\
&+\int_s^1\left(-4ch^2\zeta\cdot sh\zeta\cdot\phi-\frac{1}{2}sh^3\zeta\cdot\phi+\frac{3}{2}ch\zeta\cdot sh^2\zeta\cdot\phi_x+ch^2\zeta\cdot sh\zeta\cdot\phi_{xx}\right)\\
&=\int_0^s\left((ch^2\xi\cdot sh\xi\cdot\phi_x)_x-(ch^3\xi\cdot+\frac{1}{2}ch\xi\cdot sh^2\xi)\phi)_x\right)\\
&+\int_s^1\left((ch^2\zeta\cdot sh\zeta\cdot\phi_x)_x-(ch^3\zeta+\frac{1}{2}ch\zeta\cdot sh^2\zeta)\phi)_x\right)\\
&=2ch^2\tfrac{1}{2}\cdot sh\tfrac{1}{2}\cdot\phi_x(s).
\end{aligned}
\end{equation}
Similarly to (68) and (69), we have
\begin{equation}
\begin{cases}
\dot p_i=0,\\
\dot q_i=0,\\
\dot s=ch^2\frac{1}{2}\cdot\sum_{j=1}^Np_jq_j.
\end{cases}
\end{equation}
Thus, (1) admits periodic peaked solution
\begin{equation*}
u_i=p_i\cdot ch(x-ct-[x-ct]-\tfrac{1}{2}),\ \ \  v_i=q_i\cdot ch(x-ct-[x-ct]-\tfrac{1}{2}),
\end{equation*}
where $p_i,q_i$ are some constants, and $c=ch^2\frac{1}{2}\cdot\sum_{j=1}^Np_jq_j$.

\section{Acknowledgments}
Zhigang Li's research is supported partially by NSFC (Grant No. 11871471, 11931017). Yuxi Hu's research is supported by NSFC(Grant No. 11701556).


\begin{thebibliography}{99}

\bibitem{BCD} H. Bahouri, J.Y. Chemin, R. Danchin,
{\em Fourier Analysis and Nonlinear Partial Differential Equations},
Springer (2011).

\bibitem{LLC} H. Li, Y. Li, Y. Chen,
{\em Bi-Hamiltonian structure of multi-component Novikov equation},
J. Nonlinear Math. Phys. {\bf 21} (2014) 509-520.

\bibitem{GX} X.G. Geng, B. Xue,
{\em An extension of integrable peakon equations with cubic nonlinearity},
Nonlinearity. {\bf 22} (2009) 1847-1856.

\bibitem{LL} N.H. Li, Q.P. Liu,
{\em On bi-Hamiltonian structure of two-component Novikov equation},
Phys. Lett. A 377 (2013) 257-261.

\bibitem{N} V.S. Novikov,
{\em Generalisations of the Camassa-Holm equation},
J. Phys. A {\bf 42} (2009) 342002.

\bibitem{DP} A. Degasperis, M. Procesi,
{\em Asymptotic Integrability Symmetry and Perturbation Theory, in: A. Degasperis, G. Gaeta (Eds.)},
World Scientific (1999) 23-37.

\bibitem{CKR} G.M. Coclite, H.H. Karlsen, N.H. Risebro,
{\em Numerical schemes for computing discontinuous solutions of the Degasperis-Procesi equation},
IMA J. Numer. Anal. {\bf 28} (2008) 20-105.

\bibitem{L} H. Lundmark,
{\em Formation and dynamics of shock waves in the Degasperis-Procesi equation},
J. Nonlinear SCI. {\bf 17} (2007) 269-298.

\bibitem{HI} D. D. Holm, R. I. Ivanov,
{\em Multi-component generalizations of the CH equation: Geometrical
aspects, peakons and numerical examples},
J. Phys. A. {\bf 43} (2010) 492001

\bibitem{CH} R. Camassa, D.D. Holm,
{\em An integrable shallow water equation with peaked solitons},
Phys. Rev. Lett. {\bf 71} (1993) 1661-1664.

\bibitem{FF} B. Fuchssteiner, A.S. Fokas,
{\em Symplectic structures, their B\"acklund transformations and hereditary symmetries},
Phys. D {\bf 4} (1981) 47-66.

\bibitem{OS} P.J. Olver, P. Rosenau,
{\em Tri-Hamiltonian duality between solitons and solitary-wave solutions having compact sup-port},
Phys. Rev. E {\bf 53} (1996) 1900-1906.

\bibitem{CI} A. Constantin, R. Ivanov,
{\em On an integrable two-component Camassa-Holm shallow water system},
Phys. Lett. A {\bf 372} (2008) 7129-7132.

\bibitem{CLZ} M. Chen, S. Liu, Y. Zhang,
{\em A two-component generalization of the Camassa-Holm equation and its solutions},
Lett. Math. Phys. {\bf 75} (2006) 1-15.

\bibitem{GL} G. Gui, Y. Liu,
{\em On the global existence and wave-breaking criteria for the two-component Camassa-Holm system},
J. Funct. Anal. {\bf 258} (2010) 4251-4278.

\bibitem{F} G. Falqui,
{\em On a Camassa-Holm-type equation with two dependent variables},
J. Phys. A: Math. Gen. {\bf 39} (2006) 327-42.

\bibitem{XQ} B. Xia, Z. Qiao
{\em Multi-component generalization of the Camassa-Holm equation},
J. Geom. Phys. {\bf 107} (2016) 35-44.

\bibitem{XQZ} B. Xia, Z. Qiao, R. Zhou,
{\em A synthetical two-component model with peakon solutions},
Stud. Appl. Math. {\bf 135} (2015) 248-276.

\bibitem{SQQ} J.F. Song, C.Z. Qu, Z.J. Qiao,
{\em A new integrable two-component system with cubic nonlinearity},
J. Math. Phys. {\bf 52} (2011) 013503.

\bibitem{F} A. Fpkas,
{\em On a class of physically important integrable equations},
Physica D {\bf 87} (1995) 145-150.

\bibitem{F} B. Fuchssteiner,
{\em Some tricks from the symmetry-toolbox for nonlinear equations: Generalizations of the Camassa-Holm equation},
Physica D {\bf 95} (1996) 229-243.

\bibitem{ZY} Z. Zhang, Z. Yin.
{\em Well-posedness, global existence and blow-up phenomena for an
integrable multi-component Camassa-Holm system},
Nonlinear Anal. {\bf 142} (2016) 112-133.

\bibitem{KP} T. Kato, G. Ponce,
{\em Commutator estimates and the Euler and Navier-Stokes equations},
Comm. Pure Appl. Math. {\bf 41} (1988) 891-907.

\bibitem{WY} X. Wu, Z. Yin,
{\em A note on the Cauchy problem of the Novikov equation},
Appl. Anal. {\bf 92} (2013) 1116-1137.

\bibitem{WY} X. Wu, Z. Yin,
{\em Well-posedness and global existence for the Novikon equation},
Annali Sc. Norm. Sup. Pisa:XI (2012) 707-727.

\bibitem{MMT} Y. Mi, C. Mu, W. Tao,
{\em On the Cauchy problem for the two-component Novikov equation},
Adv. Math. Phys. {\bf 2013} (2013) 810725.

\bibitem{MGL} Y. Mi, B. Guo, T, Luo,
{\em Well-posedness and analyticity of the Cauchy problem for the multi-component Novikov equation},
Monatshefte fš¹r Mathematik, (2019) 1-29.

\bibitem{W} X.L. Wu,
{\em On the Cauchy problem of a trhee-component Camassa-Holm equations},
Discrete Contin. Dyn. Syst. {\bf 36} (2016) 2827-2854.

\bibitem{D} R. Danchin,
{\em A note on well-posedness for Camassa-Holm equation},
J. Differential Equations {\bf 192} (2003) 429-444.

\end{thebibliography}
\end{document}